\documentclass[11pt,a4paper,reqno]{amsart}

\usepackage{graphicx} 

\usepackage{graphicx}
\usepackage{amsmath}
\usepackage{amssymb}
\usepackage{pst-node}

\newtheorem{thm}{Theorem}[section]

\newtheorem{lem}[thm]{Lemma}

\newtheorem{cor}[thm]{Corollary}
\newtheorem*{cor*}{Corollary}
\newtheorem{question}[thm]{Question}

\newtheorem{defn}[thm]{Definition}

\theoremstyle{remark}

\usepackage[shortlabels]{enumitem}

\newcommand{\cref}[1]{Corollary \ref{#1}}

\newcommand{\map}[4]{#1:#2^{#3}\rightarrow#4}
\newcommand{\inner}[2]{\langle#1,#2\rangle}

\newcommand{\N}{\mathbb{N}} 
\newcommand{\Z}{\mathbb{Z}} 
\newcommand{\C}{\mathbb{C}} 

\newcommand{\R}{\mathbb{R}}

\renewcommand{\S}{\mathbb{S}}
\newcommand{\RP}{\mathbb{RP}} 

\newcommand{\CP}{\mathbb{CP}} 
\newcommand{\HP}{\mathbb{HP}} 
\newcommand{\OP}{\mathbb{OP}}

\newcommand{\HH}{\mathbb{H}}
\newcommand{\F}{\pmb{F}}

\usepackage[sorting=nyt,style=numeric,citestyle=numeric]{biblatex} 

\DeclareFieldFormat[article, inbook]{title}{#1,} 
\DeclareFieldFormat[article]{pages}{#1}
\DeclareFieldFormat{journaltitle}{#1\isdot,}
\DeclareDelimFormat[bib]{nametitledelim}{: }
\renewbibmacro{in:}{}

\renewbibmacro*{issue+date}{%
  \ifboolexpr{test {\iffieldundef{year}}
              and test {\iffieldundef{issue}}}
    {}
    {\printtext[parens]{%
       \printfield{issue}%
       \setunit*{\addspace}%
       \usebibmacro{date}}}%
  \newunit}

\usepackage{blindtext}
\usepackage{hyperref}

\title[Oriented Zoll families in CROSSes]{Oriented Zoll families of minimal hypersurfaces in $\CP^n$ and $\HP^n$}
\author{Diego Guajardo}
\address{D. Guajardo: ICMC - USP \\ São Carlos SP 13566-590 Brazil}
\email{diegonguajardo@usp.br}

\thanks{D. G. was financed, in part, by the São Paulo Research Foundation (FAPESP), Brasil. Grant 2024/03493-2 and 2025/05348-2.}

\begin{document}

\begin{abstract}
    Let $M$ be either $\CP^n$ or $\HP^n$.
	We prove that the family $\overline{\mathcal{F}}=\{\Sigma_\sigma\}_{\sigma\in M}$ of minimal geodesic spheres in $M$ characterizes the Fubini-Study metric $\overline{g}$, that is, $\overline{g}$ is the only metric that makes these real hypersurfaces minimal.
    This result is related to the recent theory of Zoll families of minimal hypersurfaces $(M,g,\mathcal{F}=\{\Sigma_\sigma\})\in\mathcal{Z}$.

    It is conjectured that the only Zoll metric in $\CP^n$ is $\overline{g}$.
    We present evidence that suggests a similar rigidity of $(\CP^n,\overline{g},\overline{\mathcal{F}})$ in this context of minimal hypersurfaces.
    Namely, we show that a direct adaptation of Guillemin's Theorem for conformal deformations of Zoll metrics in $\S^2$ is not possible for $(\CP^n,\overline{g},\overline{\mathcal{F}})$, when $p=2n-1$ is a prime number.
\end{abstract}
\maketitle
\section{Introduction}

Consider a connected Riemannian manifold $(M^n,g)$.
We say that the metric $g$ is \textit{Zoll} if all geodesics are closed, simple, and of the same length.
For example, the round metric $\overline{g}$ on the sphere $\S^n$ or on the real projective space $\RP^n$ is a Zoll metric. 
In fact, the complex projective space $\CP^n$, the quaternionic projective space $\HP^n$, and the Cayley plane $\OP^2$ also have a natural Zoll metric $\overline{g}$.
For simplicity, we call $\overline{g}$ the Fubini-Study metric in all cases.
These manifolds form the family of the CROSSes (compact rank-one symmetric spaces).

Any manifold $M$ with a Zoll metric shares topological properties with one of the CROSSes. 
For example, $M$ must be compact and it is either simply connected or $\pi_1(M)=\Z/2\Z$.
Remarkably, Bott-Samelson Theorem asserts that the cohomology ring of $M$ coincides with one of the CROSSes (see Theorem 7.23 in \cite{Bes}).

Notably, $\S^n$ has many Zoll metrics. 
At the beginning of the $20^{th}$-century, O. Zoll \cite{Zoll} found spheres of revolution in $\R^3$ with this property. 
Later, A. Besse made this construction intrinsic, and A. Weinstein found Zoll metrics for spheres of any dimension (see \cite{Bes} for a general discussion of the topic). 

In contrast to spheres, the only Zoll metrics in $\RP^n$ are the round ones; L. Green \cite{Gre} proved this for $n=2$, and the case $n>2$ was solved by M. Berger (see \cite{Bes} Corollary D.2 with Theorem 7.23).
It is conjectured that the same rigidity holds for the other projective spaces.

Recently, L. Ambrozio, F. Marques, and A. Neves \cite{AMN} generalized the notion of Zoll metrics in the context of minimal hypersurfaces. 
Namely, given a connected Riemannian manifold $(M^n,g)$, a \textit{Zoll family of minimal hypersurfaces} is a set $\mathcal{F}=\{\Sigma_\sigma\}$ of minimal hypersurfaces satisfying
\begin{itemize}
	\item[$a)$] For every point $(p,\pi)\in Gr_{n-1}(M)$ in the Grassmannian of (unoriented) tangent hyperplanes, there exists a unique $\Sigma_\sigma$ in the family such that $T_p\Sigma_\sigma=\pi$; and
	\item[$b)$] The map $(p,\pi)\in Gr_{n-1}(M)\mapsto \Sigma_\sigma$ is smooth.
\end{itemize} 
Such a family induces a foliation $\hat{\mathcal{F}}:=\{\hat{\Sigma}_\sigma\}$ in $Gr_{n-1}(M)$; $\hat{\Sigma}_\sigma\subseteq Gr_{n-1}(M)$ is the lift of $\Sigma_\sigma$ via $x\in\Sigma_\sigma\rightarrow(x,T_x\Sigma)\in Gr_{n-1}(M)$.
This is analogous to Zoll metrics; the geodesic flow of a Zoll metric induces a foliation with closed leaves in $Gr_{1}(M)$, the Grassmannian of (unoriented) lines.

There are two natural examples of manifolds with Zoll families of minimal hypersurfaces: $\S^n$ and $\RP^n$ with the round metrics and the family of equators $\overline{\mathcal{F}}:=\{\Sigma_\sigma\}_{\sigma\in\RP^n}$.
There are non-trivial metrics in $\S^n$ and $\RP^n$ with Zoll families of minimal hypersurfaces (see \cite{AmbrozioMinimalEquators}, \cite{AMN}, and \cite{AGequiv}).
This contrasts with the case of geodesics; the only Zoll metrics in $\RP^n$ are the round ones. 

Recently, G. Martins \cite{GustavoZollcharact} proved that if $(M^n,g)$ has a Zoll family of closed minimal hypersurfaces $\mathcal{F}=\{\Sigma_\sigma\}_{\sigma\in P}$, then $M^n$ is diffeomorphic to $\S^n$ or $\RP^n$.
Moreover, $\Sigma_\sigma\cong\S^{n-1}$ in the first case, $\Sigma_\sigma\cong\RP^{n-1}$ in the second, and $P$ is homeomorphic to $\RP^n$.
Hence, the existence of such a Zoll family imposes a more rigid structure (from the topological point of view) than the classical notion of Zoll metric. 
In particular, the remaining CROSSes do not have a metric with such a family.

In this work, we relax the definition of a Zoll family of minimal hypersurfaces, with this modification, any CROSS $M$ has a natural family $\overline{\mathcal{F}}=\{\Sigma_{\sigma}\}_{\sigma\in M}$ of minimal hypersurfaces which is Zoll-like.

We say that a family $\mathcal{F}=\{\Sigma_\sigma\}$ of orientable minimal hypersurfaces in $M$ is an \textit{oriented Zoll family of minimal hypersurfaces} if it naturally induces a foliation in the Grassmannian of oriented hyperplanes $Gr^+_{n-1}(M)$ (instead of $Gr_{n-1}(M)$).
In particular, through any hyperplane $\pi\subseteq T_xM$ there are two (not necessarily distinct) minimal hypersurfaces tangent to $\pi$; one for each orientation.
Observe that the geodesic flow of a Zoll metric also induces a foliation in $Gr_1^+(M)$, the Grassmannian of oriented lines.

Let $(M,\overline{g})$ be a CROSS. Assume first that $M\neq\RP^n$.
As we will see, there is a unique number $d_M\in(0,\pi/2]$ such that, for $\sigma\in M$,
$$\Sigma_\sigma:=\{x\in M: dist_{\overline{g}}(\sigma,x)=d_M\}\subseteq(M,\overline{g}),$$ 
is a minimal hypersurface. 
Each $\Sigma_\sigma$ has a natural orientation; the radial direction pointing to $\sigma$.
The family $\overline{\mathcal{F}}=\{\Sigma_\sigma\}_{\sigma\in M}$ is an oriented Zoll family of minimal hypersurfaces.
We call this family the \textit{canonical family of minimal geodesic spheres}.
We must be more careful for $M=\RP^n$, since the equator orthogonal to $\sigma\in\RP^n$ is not orientable for $n$ odd. In this case, we consider $\Sigma_\sigma$ as the double cover of such equator.

When $M=\S^n$ or $\RP^n$, the foliation of $Gr_{n-1}^+(M)$ induced by $\overline{\mathcal{F}}$ has a natural compatibility with the involution $(p,\pi)\rightarrow(p,-\pi)$ of $Gr_{n-1}^+(M)$.
This condition recovers the definition of Zoll families discussed above.

Suppose that $(M^n,g)$ has a Zoll family of (oriented) minimal hypersurfaces $\mathcal{F}=\{\Sigma_\sigma\}$.
T. Hangan \cite{HanganMinimalHyperplanes} presented a strategy to find non-trivial metrics in $M^n$ with such a family.
Namely, fix the family $\mathcal{F}$ and find the metrics that make the hypersurfaces of $\mathcal{F}$ minimal.
With this, T. Hangan found non-flat metrics in the Euclidean space $\R^n$ with minimal affine hyperplanes. 
This strategy was adapted to $\S^n$ (see survey \cite{AmbrozioMinimalEquators}).
In particular, there are non-trivial metrics in $\S^n$ for which all equators are minimal. 
Surprisingly, these metrics are antipodally symmetric. 
Hence, $\RP^n$ also has non-trivial metrics with all equators being minimal.

The first result of this work analyzes Hangan's strategy for the spaces $\CP^n$ and $\HP^n$.
In contrast to $\RP^n$, the minimality of the elements of $\overline{\mathcal{F}}$ characterizes the metric.

\begin{thm}\label{Thm A CPn}
    Let $M$ be either $\CP^n$ or $\HP^n$ ($n\geq2$), let $\overline{g}$ bet the Fubini-Study metric, and consider the canonical family of minimal geodesic spheres $\overline{\mathcal{F}}=\{\Sigma_\sigma\}$. 
    If $g$ is a Riemannian metric such that $\Sigma_\sigma\subseteq (M,g)$ is a minimal hypersurface for all $\Sigma_\sigma\in\overline{\mathcal{F}}$, then $g=\lambda\overline{g}$ for some constant $\lambda\in\R_{>0}$. 
\end{thm}

From a technical point of view, Hangan's strategy for $(M^n,\overline{g},\overline{\mathcal{F}})$ requires dealing with quadratic Killing tensors $k$, that is, tensors that satisfy
$$\overline{\nabla}k(X,X,X)=0,\quad\forall X\in TM,$$
where $\overline{\nabla}$ is the Levi-Civita connection associated with $\overline{g}$.
As shown in \cite{MatveevNikolayevskyKillingNODecomposable}, the space of such tensors is fundamentally different for each CROSS.
It is not clear whether the last theorem also holds for $\OP^2$ due to this discrepancy.

We comment that L. Luzzi in \cite{Luzzi} used Hangan's strategy to study the families of projective subspaces $\CP^{k}$ in $\CP^n$ ($k<n$).
Similar questions can be asked in $\HP^n$ and $\OP^2$; these spaces have natural families of totally geodesic submanifolds (but real codimension bigger than one).

\text{ }

P. Funk \cite{Funk} presented a perturbative strategy to produce Zoll metrics on $\S^2$; search for a $1$-parameter family of Zoll metrics $g_t=e^{2\psi_t}\overline{g}$, where $\psi_t\in C^\infty(\S^2)$ and $\psi_0\equiv1$. 
He showed a necessary condition; the function $f:=\frac{d}{dt}\big|_0\psi_t\in C^{\infty}(\S^2)$ must be odd. 
Sixty years later, V. Guillemin \cite{Guillemin} proved the converse.  
Due to obstruction shown by K. Kiyohara \cite{Kiyohara}, it is not clear how to generalize this result to $\S^n$ for $n\geq3$.

In \cite{AMN}, they generalize Guillemin's result for the minimal context; given an odd function $f\in C^\infty(\S^n)$ there is a $1$-parameter family of metrics $g_t=e^{2\psi_t}\overline{g}$ in $\S^n$ with a Zoll family of minimal hypersurfaces $\mathcal{F}_t=\{\Sigma_\sigma^t\}$ (smooth in $t$) with $\psi_0\equiv0$ and $\frac{d}{dt}\big|_{t=0}\psi_t=f$.
The author of this work and L. Ambrozio in \cite{AGequiv} used a similar strategy to find Euclidean examples with a prescribed group of symmetry $G\subseteq O(n+1)$. 

It is natural to ask for an adaptation of the results discussed before for the remaining CROSSes, for example, to $(\CP^n,\overline{g},\overline{\mathcal{F}})$.
In particular, we have the following natural question:

\begin{question}\label{question exist one parameter family}
    Is there a smooth family of Riemannian metrics in $\CP^n$, say $g_t$ with $t\in(-\varepsilon,\varepsilon)$, having an oriented Zoll family of minimal hypersurfaces $\mathcal{F}_t=\{\Sigma_{\sigma}^t\}$ (smooth in $t$) with $(g_t,\mathcal{F}_t)|_{t=0}$ being the canonical family $(\overline{g},\overline{\mathcal{F}})$?
\end{question}
If such a family exists, the area of $\Sigma_\sigma^t\in\mathcal{F}_t$ does not depend on $\sigma\in\CP^n$ (first variation formula).
Moreover, rescaling $g_t$ for each $t$, we can assume that the area also does not depend on $t$. 
When this holds, we write $g_t\in\mathcal{Z}$.

The next result shows that a direct adaptation of the results of \cite{AMN} for a conformal deformation $g_t=e^{2\psi_t}\overline{g}\in\mathcal{Z}$ is not possible, for example in $\CP^2$.
Indeed, the proof given in \cite{AMN} requires non-trivial information of first order.

\begin{thm}\label{Thm B for CPn with 2n-1 prime}
    Suppose that $\CP^n$ has a $1$-parameter family of Riemannian metrics $g_t=e^{2\psi_t}\overline{g}$ with $g_0=\overline{g}$ and $g_t\in\mathcal{Z}$.
    If $p=2n-1$ is a prime number then $\frac{d}{dt}\big|_{0}\psi_t\equiv0$.
\end{thm}

The hypothesis $p=2n-1$ being a prime number is used to simplify a non-trivial number theory problem. The author suspects that the last result is valid for any $n\in\N$.

Theorem 1.5 of \cite{AGequiv} gives metrics $g_t$ on $\RP^n$ with $g_t\in\mathcal{Z}$.
However, $\frac{d}{dt}\big|_{t=0}g_t$ is orthogonal to the conformal directions.
In particular, Theorem \ref{Thm B for CPn with 2n-1 prime} does not exclude the possibility of applying this strategy to $\CP^n$; it only excludes the conformal directions.
However, in analogy to the spherical case, Hangan's strategy should integrate directions orthogonal to the conformal directions (see Remark 4.22 of \cite{AGequiv}). 
In this sense, Theorem \ref{Thm A CPn} suggests that it is also not possible to produce metrics in these directions.

Recall that it is conjectured that the only Zoll metric in $\CP^n$ is $\overline{g}$ (up to scaling).
Theorems \ref{Thm A CPn} and \ref{Thm B for CPn with 2n-1 prime} suggest a similar rigidity in this context of minimal hypersurfaces.

\text{ }

The organization of this paper is as follows.
Section \ref{section preliminaries} introduces the basic concepts and notation used in this paper.
We prove Theorem \ref{Thm A CPn} for $\CP^n$ in Section \ref{section Complex projective space}.
In Section \ref{section Quaternionic projective space}, we analyze the proof for $\HP^n$.
We remark that this case requires a more elaborate algebraic argument than in the case of $\CP^n$. 
For organizational purposes, we left this technical part in the Appendix; Section \ref{section appendix}.
Finally, in Section \ref{Section conformal deformations of CPn} we prove Theorem \ref{Thm B for CPn with 2n-1 prime}.

\text{ }

\textbf{AI disclosure.} Gemini Plus was used to verify some computations and for spell-checking.

\section{Preliminaries and notations}\label{section preliminaries}
In this section, we introduce the notation and conventions that will be used in this work.

\text{ }

We denote by $\S^n=\{x\in\R^{n+1}:|x|=1\}$ the $n$-sphere, and let $\overline{g}$ be the Riemannian metric induced from the ambient space.

Consider the Hopf fibrations $\map{\pi}{\S}{2n+1}{\CP^n}$ and $\map{\pi}{\S}{4n+3}{\HP^n}$. 
In the quaternionic case, we use the convention that the action of $Sp(1)=\S^3$ on $\S^{4n+3}$ is by multiplication on the right. 
By forcing $\pi$ to be a Riemannian submersion, we induce Riemannian metrics $\overline{g}$ on $\CP^n$ and $\HP^n$.
We call $\overline{g}$ the \textit{Fubini Study metric}, and denote by $\overline{\nabla}$ the associated Levi-Civita connection.

The Cayley plane $\OP^2$ and its Fubini-Study metric $\overline{g}$ are defined in \cite{GotohTCayley}.
We normalize the metric so that the injectivity radius is $\frac{\pi}{2}$.

Let $(M,\overline{g})$ be either $\S^n$, $\CP^n$, $\HP^n$, or $\OP^2$ with the Fubini-Study metric.
For $\sigma\in M$, let
$$\Sigma_\sigma:=\{x\in M: dist_{\overline{g}}(\sigma,x)=d\}\subseteq(M,\overline{g}),$$
where $d\in\R_{>0}$ is given by
$$d=d_M=\left\{ \begin{array}{lr}
         \frac{\pi}{2}&  \text{ if }\hspace{2mm}M=\S^n\hspace{1mm},\\
         \arccos\Big(\sqrt{\frac{1}{2n}}\Big)&  \text{ if }M=\CP^n,\\
         \arccos\Big(\sqrt{\frac{3}{4n+2}}\Big)&  \text{ if }M=\HP^n,\\
         \arccos\Big(\sqrt{\frac{7}{15}}\Big)&  \text{ if }M=\OP^2.
    \end{array}\right.$$
Since $d_M<inj(M,\overline{g})$, $\Sigma_\sigma\subseteq M$ is a smooth sphere of codimension one. 
As we will see, $d_M$ is the only number in $(0,\frac{\pi}{2}]$ that makes $\Sigma_\sigma\subseteq (M,\overline{g})$ a minimal hypersurface (for $M=\OP^2$ see \cite{GotohTCayley}).
We call $\overline{\mathcal{F}}:=\{\Sigma_\sigma\}_{\sigma\in M}$ the \textit{canonical family of minimal geodesic spheres} of $(M,\overline{g})$.

Suppose that $(M^n,g)$ is a connected Riemannian manifold of dimension $n$, and $Gr_{n-1}^+(M)$ is the Grassmannian of oriented hyperplanes.

\begin{defn}
    \normalfont A family of oriented minimal hypersurfaces $\mathcal{F}=\{\Sigma_\sigma\}_\sigma$ in $(M^n,g)$ is an \textit{oriented Zoll family of minimal hypersurfaces} if it induces a smooth foliation in $Gr^+_{n-1}(M)$.
\end{defn}
If $(M,\overline{g})$ is a CROSS, the canonical family $\overline{\mathcal{F}}$ is an oriented Zoll family of minimal surfaces (see Lemma \ref{lem zoll family} for the case $M=\CP^n$; the cases $\S^n$, $\HP^n$, and $\OP^2$ are analogous).

The case $M=\RP^n$ is more interesting.
If $n$ is even, then we define $\Sigma_\sigma$ as the equator orthogonal to $\sigma\in\RP^n$.
Since $\Sigma_\sigma\cong\RP^{n-1}$ is orientable, it induces two leaves in $Gr_{n-1}^+(\RP^n)$ simply by choosing orientation. 
If $n$ is odd we pick an orientation for $\RP^n$, and consider 
$$\Sigma_\sigma:=\big\{v\in T_\sigma \RP^n:|v|=\frac
{\pi}{2}\big\}\xrightarrow{\exp_\sigma(\cdot)}\RP^n$$ the double cover of the equator orthogonal to $\sigma$, since $\Sigma_\sigma$ is oriented, it induces a single leaf in $Gr_{n-1}^+(\RP^n)$.

Clearly, in the cases of $M=\S^n,\RP^n$, the foliation in $Gr_{n-1}^+(M)$ is compatible with the action $(p,\pi)\rightarrow(p,-\pi)$.
This compatibility allows us to work with the definition given in \cite{AMN}. 
For the remaining CROSSes, the foliation is not compatible with this involution.

In many computations, we have to deal with two metrics: an abstract Riemannian metric $g$ and the Fubini-Study metric $\overline{g}$.
Whenever necessary, we make any dependence of the metric explicit. 
For example, $Hess_{\overline{g}}f$ and $\Delta_{\overline{g}}f$ denote the Hessian and Laplacian of a function $f$ with respect to $\overline{g}$. 
On the other hand, $Hess_gf$, $\nabla^g_XY$, and $tr_g(T)$ are (respectively) the Hessian, the Levi-Civita connection, and the trace of a tensor, all computed with respect to $g$.

\section{The complex projective space case}\label{section Complex projective space}
The goal of this section is to prove Theorem \ref{Thm A CPn} for $\CP^n$. 
Many ideas of this section work for the case $\HP^n$.

\text{ }

Given $\sigma\in\CP^n$ and $d\in(0,\pi/2)$, consider the geodesic ball
$$\Sigma_\sigma=\Sigma_\sigma(d):=\{x\in\CP^n:dist_{\overline{g}}(x,\sigma)=d\}\subseteq\CP^n.$$
Alternatively, consider the Hopf fibration $\pi:\S^{2n+1}\rightarrow\CP^n$ and lift $\Sigma_\sigma$ to obtain a torus $\hat{\Sigma}_\sigma:=\pi^{-1}(\Sigma_\sigma)\subseteq\S^{2n+1}$. 
Namely, for $r\in(0,1)$, consider
\begin{equation}\label{Sigma-hat, torus}
    \hat{\Sigma}_\sigma=\hat{\Sigma}_\sigma(r):=\{(x+y)\in\S^{2n+1}:x\in\sigma,y\in\sigma^{\perp},|x|^2=r^2=1-|y|^2\},
\end{equation}
then $\hat{\Sigma}_\sigma\cong\S^1(r)\times\S^{2n-1}(\sqrt{1-r^2})$ and $\hat{\Sigma}_\sigma(r)=\pi^{-1}(\Sigma_\sigma(d))$ for $r=\cos(d)$ (we always assume this equality).
In this case, $\pi:\hat{\Sigma}_\sigma\rightarrow\Sigma_\sigma$ is a Riemannian submersion.
Moreover, since the fibers of this map are closed geodesics of $\S^{2n+1}$, the mean curvatures are related by $H_{\Sigma_\sigma}\circ\pi=H_{\hat{\Sigma}_\sigma}$.

The normal vector at $p=(x+y)\in \hat{\Sigma}_\sigma\subseteq\S^{2n+1}$ is given by 
\begin{equation}\label{eq normal vector of clifford torus in CPn}
    \hat{N}_p=\frac{\sqrt{1-r^2}}{r}x-\frac{r}{\sqrt{1-r^2}}y\in T_p\S^{2n+1},
\end{equation}
so its principal curvatures are $\frac{\sqrt{1-r^2}}{r}$ with multiplicity $1$ and $-\frac{r}{\sqrt{1-r^2}}$ with multiplicity $2n-1$. In particular, $\hat{\Sigma}_\sigma\subseteq\S^{2n+1}$ and $\Sigma_\sigma\subseteq\CP^n$ have constant mean curvature $H=\frac{1-2nr^2}{r\sqrt{1-r^2}}$.

Since $\Sigma_\sigma$ is a geodesic sphere centered at $\sigma\in\CP^n$, we give to $\Sigma_\sigma$ the orientation induced by the normal vector pointing to $\sigma$. 
Equivalently, this normal vector lifts horizontally to the vector defined in (\ref{eq normal vector of clifford torus in CPn}).

From now on, we will always work with $\cos(d)=r=\sqrt{\frac{1}{2n}}$.
This value makes $\Sigma_\sigma\subseteq(\CP^n,\overline{g})$ a minimal (real) hypersurfaces.

\begin{lem}\label{lem zoll family}
    The family $\mathcal{F}=\{\Sigma_\sigma\}_{\sigma\in\CP^n}$ determines an oriented Zoll family of minimal hypersurfaces in $(\CP^n,\overline{g})$.
\end{lem}
\begin{proof}

Given $(x,\pi)\in Gr^+_{n-1}(M)$, we need to show that there is a unique $\Sigma_\sigma$ with $T_x\Sigma_\sigma=\pi$ and the same orientation.
    
Consider an oriented hyperplane $\pi\subseteq T_x\CP^n$, and let $v\in T_x\CP^n$ orthogonal to $\pi$ with $\overline{g}(v,v)=1$. 
This vector is uniquely determined up to a sign, so we choose the one compatible with the orientations of $\pi$ and $T_x\CP^n$. 
In this case, let $\gamma_v(t)$ be the geodesic with $\gamma'(0)=v$ and consider $\sigma:=\gamma_v(d)\in\CP^n$ where $\cos(d)=r=\sqrt{\frac{1}{2n}}$. 
Observe that $x\in\Sigma_\sigma$ since 
$$d=\arccos\Big(\sqrt{\frac{1}{2n}}\Big)<\pi/2=inj(\CP^n,\overline{g}),$$ 
so the geodesic is minimizing, which proves that $x$ belongs to this geodesic sphere.
Finally, observe that $\pi=T_x\Sigma_\sigma$ by Gauss Lemma (radial geodesics are orthogonal to the geodesic spheres when the radius is smaller than the injectivity radius) and with the correct orientation.
\end{proof}


To prove Theorem \ref{Thm A CPn} for $\CP^n$, it is easier to consider $\Sigma_\sigma\subseteq\CP^n$ as a level set of a function.
Since there are many such functions, we consider one of the most natural choices; an eigenfunction of the Laplacian of $(\CP^n,\overline{g})$. 

For each $\sigma\in\CP^n$ and $x\in\C^{n+1}$, we denote  $$x=x_\sigma+x_{\sigma^\perp}\in\sigma\oplus\sigma^\perp=\C^{n+1},$$ 
the orthogonal decomposition of $x$ into $x_\sigma\in\sigma$ and $x_{\sigma^\perp}\in\sigma^\perp$. 
In this case, we denote by $Q_\sigma:\C^{n+1}\times\C^{n+1}\rightarrow\R$ the symmetric operator
$$Q_\sigma(X,Y)=n\inner{X_\sigma}{Y_\sigma}-\inner{X_{\sigma^\perp}}{Y_{\sigma^\perp}}.$$

Let $q_\sigma:\S^{2n+1}\rightarrow\R$ be the map $q_\sigma(x)=Q_\sigma(x,x)$. 
Since $q_\sigma$ is the restriction of a harmonic and homogeneous polynomial, $q_\sigma$ defines an eigenfunction of the Laplacian of $(\S^{2n+1},\overline{g})$.
Moreover, $q_\sigma$ is $\S^1$-invariant and thus induces a map in $\CP^n$ that we still denote by $q_\sigma$.
This induced map is an eigenfunction of the Laplacian of $(\CP^n,\overline{g})$ and $\Sigma_\sigma=q_{\sigma}^{-1}(c_n)$, where $c_n:=-\frac{n-1}{2n}$.
The next formulas will be used in the future without further mention.
\begin{lem}\label{lem dq and Hq}
    The derivative and the Hessian of $q_\sigma$ at $x\in\CP^n$ are given by 
    $$(dq_\sigma)_x(v)=2Q_\sigma(\hat{x},\hat{v}),$$
    $$\frac{1}{2}Hess_{\overline{g}}(q_\sigma)_x(v,w)=Q_\sigma(\hat{v},\hat{w})-q_\sigma(x)\overline{g}(v,w),$$
    where $\hat{x}\in\S^{2n+1}$ is any representative of $x\in\CP^n$, and $\hat{v},\hat{w}\in T_{\hat{x}}\S^{2n+1}$ are the horizontal lifts of $v,w\in T_x\CP^n$.
\end{lem}
\begin{proof}
    The first formula follows directly by differentiating $q_\sigma(\hat{\gamma}_{\hat{v}}(t))$, where $\hat{\gamma}_{\hat{v}}:\R\rightarrow\S^{2n+1}$ is a horizontal geodesic with $\hat{\gamma}_{\hat{v}}'(0)=\hat{v}$. 
    For the second identity, consider $V,W$ local vector fields extending $v,w\in T_x\CP^n$, and let $\hat{V},\hat{W}$ be their horizontal lifts to the sphere.
    Also, denote by $\widehat{\overline{\nabla}_VW}$ the horizontal lift of $\overline{\nabla}_{V}W$ to $\S^{2n+1}$.
    It is well known that
    $$\overline{\nabla}_{\hat{V}}\hat{W}=\widehat{\overline{\nabla}_VW}+\overline{g}(X,JY)R,$$
    where, on the left-hand side, we are using the connection of the round sphere $(\S^{2n+1},\overline{g})$, $J:T\CP^n\rightarrow T\CP^n$ is the complex structure, and $R_p:=ip\in T_p\S^{2n+1}$ is the Reeb vector field of $\S^{2n+1}$. The Hessian is then given by
    \begin{align*}
        \frac{1}{2}Hess_{\overline{g}}q_\sigma(v,w)&=\frac{1}{2}VW(q_\sigma)-\frac{1}{2}\overline{\nabla}_VW(q_\sigma)=\frac{1}{2}V(dq_\sigma(W))-\frac{1}{2}dq_\sigma(\overline{\nabla}_VW)\\
        &=\hat{V}(Q_{\sigma}(\hat{x},\hat{W}))-Q_{\sigma}(\hat{x},\widehat{\overline{\nabla}_{V}W})\\
        &=Q_\sigma(\hat{V},\hat{W})+Q_\sigma(\hat{x},\overline{\nabla}_{\hat{V}}\hat{W}-\inner{V}{W}\hat{x})-Q_{\sigma}(\hat{x},\widehat{\overline{\nabla}_{V}W})\\
        &=Q_\sigma(\hat{v},\hat{w})-q_\sigma(x)\overline{g}(v,w),
    \end{align*}
    where in the last step we used that $Q_\sigma(\hat{x},i\hat{x})=0$ for any $x$.
\end{proof}

Consider now an arbitrary metric $g$ in $\CP^n$, we will study the condition for $\Sigma_\sigma\subseteq(\CP^n,g)$ being minimal for each $\sigma\in\CP^n$.
For $x\in\Sigma_\sigma$, let 
$$N_\sigma(x):=\frac{\nabla^gq_\sigma(x)}{|\nabla^g q_\sigma(x)|_g}\in T_x\CP^n,$$
the unit normal vector (with respect to $g$) of $\Sigma_\sigma=q_\sigma^{-1}(c_n)$ at $x$.
To simplify the notation, we will always omit the point $x$.

The second fundamental form of $\Sigma_\sigma\subseteq(\CP^n,g)$ is given by
$$A_g(X,Y)=g(\nabla^g_X N_\sigma,Y)=\frac{1}{|\nabla^gq_\sigma|_g}Hess_{g}q_\sigma(X,Y).$$
In this case, the mean curvature of $\Sigma_\sigma\subseteq(\CP^n,g)$ is given by
\begin{equation}\label{formula of H in terms of q}
    H=\frac{1}{|\nabla^gq_\sigma|_g}\Big(\Delta_gq_\sigma-Hess_{g}q_\sigma(N_\sigma,N_\sigma)\Big).
\end{equation}

On the other hand, if we consider the $3$-tensor $T$ given by
$$T(X,Y,Z)=g(\nabla^g_XY-\overline{\nabla}_XY,Z),\quad X,Y,Z\in T\CP^n,$$
where $\nabla^g$ is the Levi-Civita connection of $g$, then
\begin{align*}
    Hess_gq_\sigma(X,Y)&=XY(q_\sigma)-(\nabla^g_XY)q_\sigma\\
    &=Hess_{\overline{g}}q_\sigma(X,Y)+(\overline{\nabla}_XY-\nabla^g_XY)q_\sigma\\
    &=Hess_{\overline{g}}q_\sigma(X,Y)-T(X,Y,\nabla^gq_\sigma),
\end{align*}
and so
$$\Delta_gq_\sigma=tr_g(Hess_{\overline{g}}q_\sigma)-tr^{12}_g(T)(\nabla^gq_\sigma).$$
Combining these identities with (\ref{formula of H in terms of q}), the minimal equation $\Sigma_\sigma\subseteq(\CP^n,g)$ at $x\in\Sigma_\sigma$ is equivalent to
\begin{equation}\label{H=0 case CPn}
    tr^{12}_g(T)(\nabla^gq_\sigma)-T(N_\sigma,N_\sigma,\nabla^gq_\sigma)=tr_g(Hess_{\overline{g}}q_\sigma)-Hess_{\overline{g}}q_\sigma(N_\sigma,N_\sigma),
\end{equation}
where $tr^{12}_g$ indicates the trace of the first two entries with respect to $g$.

We can compare (\ref{H=0 case CPn}) with the spherical case studied in \cite{AmbrozioMinimalEquators} and \cite{AMN}.
In their situation, the right-hand side of (\ref{H=0 case CPn}) is trivially zero due to the linear condition that determines the equators of $\S^n$.
However, $\Sigma_\sigma\subseteq\CP^n$ is determined by a quadratic equation, which imposes additional terms to the minimal equation.

Equation (\ref{H=0 case CPn}) is given at each $(x,\sigma)$ with $x\in\Sigma_\sigma$, so the strategy to simplify this equation will be to fix $x$ and vary $\sigma$.

\begin{lem}\label{lemma tilde-sigma}
    For each $x\in\Sigma_\sigma$, there exists $\tilde{\sigma}=\tilde{\sigma}(x,\sigma)\in\CP^n$ such that $x\in\Sigma_{\tilde{\sigma}}$, $(\nabla^gq_{\tilde{\sigma}})_x=-(\nabla^gq_\sigma)_x\in T_x\CP^n$, and $Hess_{\overline{g}}(q_{\tilde{\sigma}})_x=Hess_{\overline{g}}(q_\sigma)_x$.
\end{lem}
\begin{proof}
    As $x\in\Sigma_\sigma$ we have that $dist_{\overline{g}}(x,\sigma)=d$, where $d=\arccos(\sqrt{\frac{1}{2n}})<\pi/2=inj(\CP^n,\overline{g})$.
    Hence, there exists a unique $v\in T_x\CP^n$ such that $\gamma_v(t)=\exp_x(tv)$ is a minimizing geodesic for $t<\pi/2$ and $\gamma_v(d)=\sigma$. 
    The association $v\rightarrow\sigma(v)$ defines a bijection between the set of unit vectors of $T_x\CP^n$ with the elements $\sigma\in\CP^n$ such that $x\in\Sigma_\sigma$.
    
    Let $\tilde{\sigma}:=\gamma_v(-d)\in\CP^n$, that is, $\tilde{\sigma}=\sigma(-v)$.
    Observe that $dist_{\overline{g}}(x,\tilde{\sigma})=d$ since the injectivity radius of $x\in(\CP^n,\overline{g})$ is $\pi/2$, so $x\in\Sigma_{\tilde{\sigma}}$.
    We claim that $\tilde{\sigma}$ satisfies the conditions. 
    To do so, we will express $(dq_\sigma)_x$ and $Hess_{\overline{g}}(q_\sigma)_x$ in terms of $v$.

    Let $\hat{x}\in\S^{2n+1}$ be a representative of $x\in\CP^n$ and consider $\hat{v}\in T_{\hat{x}}\S^{2n+1}$ the horizontal lift of $v\in T_x\CP^n$. In this case, $\hat{\gamma}_v(t)=\cos(t)\hat{x}+\sin(t)\hat{v}$ is the horizontal lift of the geodesic $\gamma_v(t)$. 
    For any $y\in\C^{n+1}$ we have
    $$y_\sigma=\inner{y}{\hat{\gamma}_v(d)}\hat{\gamma}_{v}(d)+\inner{y}{i\hat{\gamma}_v(d)}i\hat{\gamma}_{v}(d)$$
    so, at $x\in\CP^n$, we have
    \begin{align*}
        \frac{1}{2}dq_\sigma(w)&=Q_\sigma(\hat{x},\hat{w})=n\inner{\hat{x}_\sigma}{\hat{w}_\sigma}-\inner{\hat{x}_{\sigma^\perp}}{\hat{w}_{\sigma^\perp}}\\
        &=(n+1)\inner{\hat{x}_\sigma}{\hat{w}_\sigma}\\
        &=(n+1)(\inner{\hat{x}}{\hat{\gamma}_v(d)}\inner{\hat{w}}{\hat{\gamma}_v(d)}+\inner{\hat{x}}{i\hat{\gamma}_v(d)}\inner{\hat{w}}{i\hat{\gamma}_v(d)})
    \end{align*}
    so
    \begin{equation}\label{eq dq-sigma wrt v}
        \frac{1}{2}dq_\sigma(w)=(n+1)\cos(d)\sin(d)\inner{\hat{w}}{\hat{v}},
    \end{equation}
    and by an analogous computation at $x\in\CP^n$ we get
    $$dq_{\tilde{\sigma}}(w)=(n+1)\cos(d)\sin(d)\inner{\hat{w}}{-\hat{v}}=-dq_\sigma(w),$$
    so $\nabla^gq_\sigma(x)=-\nabla^gq_{\tilde{\sigma}}(x)$.

    Similar computations show the equality of the Hessians at $x$. 
    In fact, given $u,w\in T_x\CP^n$ we have
    \begin{align*}
        \frac{1}{2}Hess_{\overline{g}}q_\sigma(u,w)&=Q_\sigma(\hat{u},\hat{w})-c_n\overline{g}(u,w)\\
        &=(n+1)\inner{\hat{u}_\sigma}{\hat{w}_\sigma}-(c_n+1)\overline{g}(u,w)\\
        =(&n+1)\sin(d)^2(\inner{\hat{u}}{\hat{v}}\inner{\hat{w}}{\hat{v}}+\inner{\hat{u}}{i\hat{v}}\inner{\hat{w}}{i\hat{v}})-(c_n+1)\overline{g}(u,w),
    \end{align*}
    so using $\sin(d)^2=\frac{2n-1}{2n}$ and $c_n=-\frac{n-1}{2n}$ we get
    \begin{equation}\label{eq Hess(q sigma)=H_V}
        \frac{1}{2}Hess_{\overline{g}}q_\sigma(u,w)=\frac{(n+1)}{2n}\Big((2n-1)\inner{\hat{u}_V}{\hat{w}_V}-\overline{g}(u,w)\Big),
    \end{equation}
    where $\hat{u}_V,\hat{w}_V$ denotes the orthogonal projections of $\hat{v}, \hat{w}$ on the $\C$-line given by $V=\C \hat{v}\subseteq\C^{n+1}$. 
    In particular, $Hess_{\overline{g}}(q_\sigma)_x$ depends only on $V$, so $Hess_{\overline{g}}(q_\sigma)_x=Hess_{\overline{g}}(q_{\tilde{\sigma}})_x$ since $v$ and $-v$ determine the same $\C$-line.
\end{proof}

As a direct consequence of the last result, we can simplify equation (\ref{H=0 case CPn}).

\begin{cor}\label{Cor H=0 in two equations}
    Using notations of Lemma \ref{lemma tilde-sigma}, $\Sigma_\sigma,\Sigma_{\tilde{\sigma}}\subseteq\CP^n$ are minimal at $x$ if and only if the following equations hold
    \begin{equation}\label{eq tr(T)(X)=T(X,X,X) CPn}
        tr^{12}_g(T)(\nabla^gq_\sigma)-T(N_\sigma,N_\sigma,\nabla^gq_\sigma)=0,
    \end{equation}
    \begin{equation}\label{eq tr(H)=H(N.N) CPn}
        tr_g(Hess_{\overline{g}}q_\sigma)-Hess_{\overline{g}}q_\sigma(N_\sigma,N_\sigma)=0.
    \end{equation}
\end{cor} 
\begin{proof}
    The minimal equation of $\Sigma_\sigma\subseteq\CP^n$ at $x\in\Sigma_\sigma$ is equivalent to (\ref{H=0 case CPn}).
    Using this equation for $x\in\Sigma_{\tilde{\sigma}}$ gives
    $$-tr^{12}_g(T)(\nabla^gq_\sigma)+T(N_\sigma,N_\sigma,\nabla^gq_\sigma)=tr_g(Hess_{\overline{g}}q_\sigma)-Hess_{\overline{g}}q_\sigma(N_\sigma,N_\sigma),$$
    and combining this with (\ref{H=0 case CPn}) gives the desired equations.
\end{proof}

We can now prove Theorem \ref{Thm A CPn} for $\CP^n$.

\begin{proof}[Proof of Theorem \ref{Thm A CPn} case $\CP^n$]
    Assume that $\Sigma_\sigma\subseteq\CP^n$ is minimal for any $\sigma\in\CP^n$.
    As we discussed before, the minimal equation $\Sigma_\sigma\subseteq\CP^n$ at $x\in\Sigma_\sigma$ is equivalent to (\ref{H=0 case CPn}). Moreover, by Corollary \ref{Cor H=0 in two equations}, we know that (\ref{eq tr(T)(X)=T(X,X,X) CPn}) and (\ref{eq tr(H)=H(N.N) CPn}) hold for any $(x,\sigma)$ wit $x\in\Sigma_\sigma$.
    
    Consider the positive function $F=F_g:\CP^n\rightarrow\R$ determined by 
    $$dV_g=F^{(2n+1)/4}dV_{\overline{g}},$$
    where $dV_g$ and $dV_{\overline{g}}$ are the respective volume forms of $g$ and $\overline{g}$. 
    Following the arguments of \cite{AMN} we can prove that equation (\ref{eq tr(T)(X)=T(X,X,X) CPn}) is equivalent to $k:=F^{-1}g$ being a Killing $2$-tensor (see Lemma \ref{lemma T(X,X,X)=g(X,X)tr(T)(X) then k Killing} of the Appendix).
    That is, $k$ satisfies 
    $$(\overline{\nabla}k)(X,X,X)=0,\quad\forall X\in T\CP^n.$$

    As observed in the proof of Lemma \ref{lemma tilde-sigma}, if $x\in\Sigma_\sigma$ there is a unique unit vector $v\in T_x\CP^n$ that determines $\sigma=\sigma(v)$ (we gave the bijection $v\rightarrow\gamma_v(d)=\sigma(v)$).
    With this bijection, (\ref{eq Hess(q sigma)=H_V}) holds for the complex line $V=\C\hat{v}$.
    To simplify notation, let 
    $H_V:T_x\CP^n\times T_x\CP^n\rightarrow\R$ the tensor 
    $$H_V(u,w)=(2n-1)\inner{\hat{u}_V}{\hat{w}_V}-\overline{g}(u,w).$$
    In this case, (\ref{eq tr(H)=H(N.N) CPn}) is equivalent to 
    \begin{equation}\label{eq tr(HV)=HV(N,N)}
        tr_k(H_{V})=H_V(N_v,N_v),
    \end{equation}
    where $N_v:=F^{1/2}N_\sigma$ is the unit normal of $\Sigma_\sigma$ with respect to the metric $k$.
    Observe that $T_x\Sigma_\sigma$ has a unique maximal complex subspace of real dimension $2n-2$, namely, $W_v:=T_x\Sigma_\sigma\cap J(T_x\Sigma_\sigma)$. Alternatively, $W_v$ correspond to $(\text{span}_\R\{\hat{x},i\hat{x},\hat{v},i\hat{v}\})^\perp\subseteq T_{\hat{x}}\S^{2n+1}$ under horizontal identification.
    In particular, there is a unique (up to sign) $N^*_v\in T_x\Sigma_\sigma$ with $k(N^*_v,N^*_v)=1$ and that is also $k$-orthogonal to $W_v$.
    Observe that $\hat{w}_V=0$ for any $w\in W_v$ (where $\hat{w}$ is the horizontal lift of $w$).
    Hence, consider a $k$-orthonormal basis in which the first elements are $\{N_v,N_v^*\}$ and the remaining belong to $W_v$.
    Using this basis, (\ref{eq tr(HV)=HV(N,N)}) is equivalent to
    \begin{equation}\label{eq tr(g)=}
        tr_k(\overline{g})=(2n-1)\inner{(\hat{N}^*_v)_V}{(\hat{N}^*_v)_V}+\inner{\hat{N}_v}{\hat{N}_v}.
    \end{equation}
    
    Observe that the association $v\rightarrow N^*_v$ can be done continuously, since $v\rightarrow N_v$ is continuous and we can use the orientation of the complex space $W_v$ to determine uniquely $N^*_v$.
    Moreover, the map  $v\rightarrow N^*_v$ is a bijection between the vectors of unit $\overline{g}$-length with the vectors of unit $k$-legth. 
    Hence, choose $v$ such that $\inner{\hat{N}^*_{v}}{\hat{N}^*_{v}}\leq\inner{\hat{N}^*_w}{\hat{N}^*_w}$ for any $w$.
    Then by (\ref{eq tr(g)=}) we get
    $$(2n-1)\inner{\hat{N}^*_v}{\hat{N}^*_v}+\inner{\hat{N}_v}{\hat{N}_v}\leq tr_k(\overline{g})=(2n-1)\inner{(\hat{N}^*_v)_V}{(\hat{N}^*_v)_V}+\inner{\hat{N}_v}{\hat{N}_v},$$
    so $(\hat{N}^*_v)_V=\hat{N}^*_v$ and $\inner{u}{u}=\inner{N^*_v}{N^*_v}=:\lambda$ for any $u\in W_v$ with $k(u,u)=1$.
    By polarization, we have that $\overline{g}|_{T_x\Sigma_{\sigma(v)}}=\lambda k|_{T_x\Sigma_{\sigma(v)}}$.
    
    Summarizing, we have shown that for any $v$ with $\lambda=\inner{N_v^*}{N_v^*}$ minimal the metrics $\overline{g}$ and $\lambda k$ coincide in $T_x\Sigma_{\sigma(v)}$ (up to scaling by $\lambda$). However, we showed that there are infinite elements that attain this minimal value, so the metrics coincide on an infinite number of distinct hyperplanes. This implies that the metrics coincide in $T_x\CP^n$ (up to scaling by $\lambda=\lambda(x)$).
    
    Hence, $\overline{g}=\lambda k$ for some positive function $\lambda:\CP^n\rightarrow\R$. However, $\overline{g}$ and $k$ are Killing $2$-tensors, so the Killing equation for both give
    \begin{align*}
        0&=\overline{\nabla}_X\overline{g}(X,X)=X(\lambda k(X,X))-2\lambda k(\overline{\nabla}_XX,X)\\
        &=d\lambda(X)k(X,X)+\lambda\overline{\nabla}_X k(X,X)=d\lambda(X)k(X,X),\quad\forall X\in\CP^n,
    \end{align*}
    which proves that $\lambda$ is constant.
    However, in this case the volume form of $g=Fk=F\lambda^{-1}\overline{g}$ is $dV_g=F^{2n}\lambda^{-2n}dV_{\overline{g}}$, so $F^{(2n+1)/4}=F^{2n}\lambda^{-2n}$ which shows that $F$ is constant and this concludes the proof.
\end{proof}

\section{The Quaternionic projective space case}\label{section Quaternionic projective space}

In this section, we prove Theorem \ref{Thm A CPn} for $\HP^n$. The sketch of the proof is similar to the case $\CP^n$, so we will skip analogous steps. 
Nevertheless, we highlight the necessary modifications.
However, the final part of the proof is more complicated than for $\CP^n$.
For this reason, we separate the final algebraic argument in Lemma \ref{lemma tr(Hv)=Hv(Nv,Nv) for HPn} of the Appendix.

\text{ }

We identify any element $\sigma\in\HP^n$ with the respective $4$-dimensional vector space of $\mathbb{H}^{n+1}:=\R^{4n+4}$ (the associated quaternionic line).
We use the convention that quaternionic lines of $\HH^{n+1}$ are closed by scalar multiplication of $\HH$ on the right.
In this case, the Hopf fibration $\map{\pi}{\S}{4n+3}{\HP^n}$ is the quotient of $\S^{4n+3}$ by the right-action of $Sp(1)=\S^3\subseteq\HH$.

For $\sigma\in\HP^n$ and $d\in(0,\pi/2)$, consider the geodesic ball
$$\Sigma_\sigma=\Sigma_\sigma(d):=\{x\in\HP^n:dist_{\overline{g}}(x,\sigma)=d\}\subseteq\CP^n,$$
with the orientation given by the normal vector that points towards $\sigma$.
Using the Hopf fibration $\map{\pi}{\S}{4n+3}{\HP^n}$, we can lift $\Sigma_\sigma$ to $\S^{4n+3}$ and obtain a torus.
More precisely, if
$$\hat{\Sigma}_\sigma=\hat{\Sigma}_\sigma(r):=\{(x+y)\in\S^{4n+3}:x\in\sigma,y\in\sigma^{\perp},|x|^2=r^2=1-|y|^2\}\subseteq\S^{4n+3}\hspace{-1mm},$$
then $\hat{\Sigma}_\sigma\cong\S^3(r)\times\S^{4n-1}(\sqrt{1-r^2})$ and $\hat{\Sigma}_\sigma(r)=\pi^{-1}(\Sigma_\sigma(d))$ for $r=\cos(d)$.

We can compute the mean curvature in the same way as we did for $\CP^n$.
Namely, $H_{\Sigma_\sigma}\circ\pi=H_{\hat{\Sigma}_\sigma}=\frac{3-r^2(4n+2)}{r\sqrt{1-r^2}}$, so
$\Sigma_\sigma\subseteq(\HP^n,\overline{g})$ is minimal for $\cos(d)=r=\sqrt{\frac{3}{4n+2}}$.
We always assume this equality.

We define $Q_\sigma:\mathbb{H}^{n+1}\times\mathbb{H}^{n+1}\rightarrow \R$ as the bilinear map 
$$Q_\sigma(X,Y)=n\inner{X_\sigma}{Y_\sigma}-\inner{X_{\sigma^\perp}}{Y_{\sigma^\perp}},$$
where $X_\sigma$ and $X_{\sigma^\perp}$ denote the orthogonal projections of $X$ into $\sigma, \sigma^\perp\subseteq\mathbb{H}^{n+1}$, respectively.
In this case, let $q_\sigma:\S^{4n+3}\rightarrow\R$ be the map given by $q_\sigma(x)=Q_\sigma(x,x)$. 
As $q_\sigma$ is invariant by the Hopf fibration, we think of $q_\sigma$ as a map of $\HP^n$.
In this case, $q_\sigma$ is harmonic in $(\HP^n,\overline{g})$ and $\Sigma_\sigma=q_\sigma^{-1}(c_n)$ where $c_n=-\frac{n-1}{4n+2}$.

The formulas of Lemma \ref{lem dq and Hq} translate to our context, but for the maps $q_\sigma$ and $Q_\sigma$ defined in this section. 
We highlight that, for $\CP^n$, we used the complex structure $J$ and the vertical vector field $R_{\hat{x}}=i\hat{x}$ that satisfy $Q_\sigma(\hat{x},i\hat{x})=0$. 
In $\HP^n$ we have three complex structures $I,J$, and $K$ that satisfy the quaternionic relations, moreover, they determine three horizontal vectors in $\S^{4n+3}$, $R^I_{\hat{x}}:=\hat{x}\cdot i$, $R^J_{\hat{x}}:=\hat{x}\cdot j$, and $R^K_{\hat{x}}:=\hat{x}\cdot k$. 
It is straightforward to verify that $Q_{\sigma}(\hat{x},R^I_{\hat{x}})=Q_{\sigma}(\hat{x},R^K_{\hat{x}})=Q_{\sigma}(\hat{x},R^K_{\hat{x}})=0$. 

Given any metric $g$ in $\HP^n$, the minimal equation of $\Sigma_\sigma\subseteq(\HP^n,g)$ is given by the same equation (\ref{H=0 case CPn}) for our maps $q_\sigma:\HP^n\rightarrow\R$.
To simplify this equation, we use an adapted version of Lemma \ref{lemma tilde-sigma}. 
Straightforward adjustments show that $\nabla^gq_{\tilde{\sigma}}(x)=-\nabla^gq_\sigma(x)$ for $x\in\Sigma_\sigma$, and 
\begin{equation}\label{eq H_V for case HPn}
    \frac{1}{2}Hess_{\overline{g}}q_\sigma(u,w)=\frac{n+1}{4n+2}\Big((4n-1)\inner{\hat{u}_V}{\hat{w}_V}-3\overline{g}(u,w)\Big),
\end{equation}
where $\hat{u},\hat{w}\in T_{\hat{x}}\S^{4n+3}$ are horizontal lifts, and $V\subseteq\HH^{n+1}$ is the quaternionic line determined by $\hat{v}\in T_{\hat{x}}\S^{4n+3}\subseteq\HH^{n+1}$.
Using this, we obtain an analogous version of Corollary \ref{Cor H=0 in two equations}.

With the discussion above, we can prove Theorem \ref{Thm A CPn} for $\HP^n$.

\begin{proof}[Proof of Theorem \ref{Thm A CPn} case $\HP^n$]
    As before, define $F=F_g:\HP^n\rightarrow\R$ by $dV_g=F^{(4n+1)/2}dV_{\overline{g}}$.
    In this case, equation (\ref{eq tr(T)(X)=T(X,X,X) CPn}) for our contexts is equivalent (as before) to the tensor $k:=F^{-1}g$ being a Killing $2$-tensor.

    Given $x\in\Sigma_\sigma$, there is a unique vector $v\in T_x\HP^n$ such that $\sigma=\gamma_v(d)$ where $d=\arccos\big(\sqrt{\frac{3}{4n+2}}\big)$ and $\gamma_v$ is the geodesic with $\gamma'(0)=v$.
    This determines a bijection $\sigma=\sigma(v)$ between the unit vectors $v$ of $T_x\HP^n$ and the elements $\sigma\in\HP^n$ with $x\in \Sigma_\sigma$. 
    The Hessian of $q_\sigma=q_{\sigma(v)}$ is determined by (\ref{eq H_V for case HPn}) and must satisfy (\ref{eq tr(H)=H(N.N) CPn}).
    Using Lemma \ref{lemma tr(Hv)=Hv(Nv,Nv) for HPn} of the Appendix to each $T_x\HP^n\cong\R^{4n}$ we get that $k=\lambda^{-1}\overline{g}$ for some positive function $\lambda:\HP^n\rightarrow\R$.
    We show that $\map{\lambda}{\HP}{n}{\R}$ is constant by using the Killing equation for $\overline{g}$ and for $k$ just as in the $\CP^n$ case. 
    Finally, the original metric $g=Fk=F\lambda^{-1}\overline{g}$ must have $F$ constant by the same argument as in $\CP^n$.
\end{proof}

\section{Conformal deformations in the complex projective space}\label{Section conformal deformations of CPn}
For simplicity, we analyze the case $\CP^n=(\CP^n,\overline{g})$, although a similar study can be done on the other projective spaces. 

\text{ }

We know that $(\CP^n,\overline{g})$ has an oriented Zoll family of minimal hypersurfaces $\mathcal{F}=\{\Sigma_\sigma\}_{\sigma\in\CP^n}$ (see Lemma \ref{lem zoll family}).
We are interested in the perturbation problem, more precisely, in Question \ref{question exist one parameter family} that, for convenience of the reader, we rewrite here: 
\begin{question}\label{question exist one parameter family, second time}
    Is there a smooth family of Riemannian metrics in $\CP^n$, say $g_t$ with $t\in(-\varepsilon,\varepsilon)$, having an oriented Zoll family of minimal hypersurfaces $\mathcal{F}_t=\{\Sigma_{\sigma}^t\}$ (smooth in $t$) with $(g_t,\mathcal{F}_t)|_{t=0}$ being the canonical family $(\overline{g},\overline{\mathcal{F}})$?
\end{question}

If such a family $(g_t,\mathcal{F}_f)$ exists, the area of $\Sigma_\sigma^t\in\mathcal{F}_t$ does not depend on $\sigma\in\CP^n$ by the first variation formula.
Moreover, since we can rescale the metric $g_t$ at each $t$, we can assume that this value also does not depend on $t$. When this condition holds, we simply write $g_t\in\mathcal{Z}$.

Just as in the classical result of Guillemin (or in the results of \cite{AMN}), we study the last question in the conformal directions. 
Namely, we assume that $g_t=e^{2\psi_t}\overline{g}\in\mathcal{Z}$ for some functions $\psi_t:\CP^n\rightarrow\R$.

For $f\in C^\infty(\CP^n)$, we define its {\it Funk transform associated to} $\overline{\mathcal{F}}=\{\Sigma_\sigma\}$, or simply its {\it Funk transform}, as the map $\F(f)\in  C^\infty(\CP^n)$ given by
$$(\pmb{F}(f))(\sigma):=\int_{\Sigma_\sigma}f(x)dA_{\overline{g}},$$
where $dA_{\overline{g}}$ is the volume form of $\Sigma_\sigma\subseteq\CP^n$.
This type of transformations have many applications; see \cite{Helgason1999Radon} for a classical reference. 

\begin{lem}\label{lem f must be in kernel of F}
    Assume that $g_t=e^{2\psi_t}\overline{g}\in\mathcal{Z}$ and let $f:=\frac{d}{dt}\big|_0\psi_t\in C^{\infty}(\CP^n)$.
    Then $\F(f)=0$, that is, $f\in\ker(\F)$.
\end{lem}
\begin{proof}
    Since the area of $\Sigma_\sigma^t$ is independent of $t$ and $\sigma$, we have that
    $$\frac{d}{dt}\Big|_{t=0}\Big(\int_{\Sigma^t_\sigma}dA_{g_t}\Big)=0.$$
    In the last equality, we have two parameters varying (the hypersurface and the metric), hence, the last derivative can be computed as follows
    $$0=\frac{d}{dt}\Big|_{t=0}\Big(\int_{\Sigma^t_\sigma}dA_{\overline{g}}\Big)+\frac{d}{dt}\Big|_{t=0}\Big(\int_{\Sigma_\sigma}dA_{g_t}\Big).$$
    The first term in the last equation vanishes by the first variation formula, and, for the second term, we observe that $dA_{g_t}=e^{(2n-1)\psi_t}dA_{\overline{g}}$, so
    $$0=(2n-1)\int_{\Sigma_\sigma}fdA_{\overline{g}}=(2n-1)\F(f)(\sigma),\quad\forall\sigma\in\CP^n$$
    which concludes the lemma.
\end{proof}

The last lemma shows that, in order to prove Theorem \ref{Thm B for CPn with 2n-1 prime}, we must understand the kernel of $\F$. To do so, we will use the spectral decomposition of $L^2(\CP^n)$.

For a compact Riemannian manifold $(M^n,g)$, let $\mathcal{H}_k(M^n,g)\subseteq C^\infty(M^n)$ be the $k^{th}$-eigenspace of the Laplacian $\Delta_{g}$.
For example, it is known that any element of $\mathcal{H}_k(\S^n,\overline{g})$ is the restriction of some homogeneous polynomial $\map{p}{\R}{n+1}{\R}$ of degree $k$.

In this section, we will work with $\mathcal{H}_k:=\mathcal{H}_k(\CP^n,\overline{g})$. 
The space $\mathcal{H}_k$ can be identified with the subset of $\mathcal{H}_{2k}(\S^{2n+1},\overline{g})$ that are $\S^1$-invariant.
To be more precise, we can use complex coordinates and use a Hodge-type decomposition 
$$\mathcal{H}_{2k}(\S^{2n+1},\overline{g})\otimes\C=\bigoplus_{l=0}^{2k}\mathcal{H}_{l,2k-l}.$$
With this decomposition, we can identify $\mathcal{H}_k\otimes\C$ with $\mathcal{H}_{k,k}$ (see Section 1 of \cite{GrinbergHarmonicsProjective} for a more detailed exposition).
More precisely, the elements of $\mathcal{H}_k$ are restrictions of homogeneous harmonic polynomials of $\R^{2n+2}$ of degree $2k$ that are $\S^1$-invariant.

As shown in Section 1 of \cite{GrinbergHarmonicsProjective}, the natural action $U(n+1)\times\mathcal{H}_k\rightarrow\mathcal{H}_k$ gives irreducible and non-equivalent representations of $U(n+1)$ for each $k$.

Observe that the Funk transform $\F:C^\infty(\CP^n)\rightarrow C^\infty(\CP^n)$ commutes with the natural action of $U(n+1)$ on $C^\infty(\CP^n)$. 
Indeed, this comes from the fact that $U(n+1)$ acts by isometries in $(\CP^n,\overline{g})$ and the family $\overline{\mathcal{F}}$ is preserved by this action (the hypersurfaces of the family are geodesic balls).
Schur's Lemma implies that there are constants $c_k=c_{k,n}\in\R$ such that 
$$\F(f_k)=c_kf_k,\quad f_k\in\mathcal{H}_k.$$
Hence, to determine $c_k$ is enough to compute $\frac{\F(f_k)(\sigma)}{f_k(\sigma)}$ for any $f_k\in\mathcal{H}_k$ and $\sigma\in\CP^n$ with $f_k(\sigma)\neq0$. The next lemma gives us a candidate $f_k\in\mathcal{H}_k$.

\begin{lem}
    Given any $k,n\in\N$, there is a homogeneous polynomial
    $$p_k(x,y)=\sum_{j=0}^ka_jx^{k-j}y^{j}\in\R[x,y],$$ 
    such that, in complex coordinates $(z_i)_i\in\C^{n+1}$, the polynomial
    $$P_k(z_0,\ldots, z_n):=p_k(|z_0|^2,|z_1|^2+\ldots+|z_n|^2),$$
    is harmonic and $\S^1$-invariant. In particular, $f_k=P_k|_{\S^{2n+1}}$ can be identified with an element of $\mathcal{H}_k$ which is constant in $\Sigma_{\sigma_0}$, where $\sigma_0=[1:0:\ldots:0]$.
\end{lem}
\begin{proof}
    The condition $\Delta P_k=0$ restricts the coefficients of $p_k$. 
    More precisely,
    \begin{align*}
        0&=\frac{1}{4}\Delta P=\partial_{z_0}\partial_{\overline{z_0}}P+\sum_{l=1}^n\partial_{z_l}\partial_{\overline{z_l}}P\\
        &=\sum_{j=0}^{k-1}a_j(k-j)^2(|z_0|^2)^{k-j-1}(|z_1|^2+\ldots+|z_n|^2)^j\\
        &\quad+\sum_{j=1}^ka_jj(n+j-1)(|z_0|^2)^{k-j}(|z_1|^2+\ldots+|z_n|^2)^{j-1}\\
        &=\sum_{j=1}^k(a_{j-1}(k-j+1)^2+a_j(n+j-1)j)(|z_0|^2)^{k-j}(|z_1|^2+\ldots+|z_n|^2)^{j-1},
    \end{align*}
    which implies
    $$0=a_{j-1}(k-j+1)^2+a_j(n+j-1)j,\quad j=1,\ldots,k.$$
    Solving the last recurrence we get that
    \begin{equation}\label{coefficient polynomial}
        a_j=a_0\frac{\binom{k}{j}\binom{n+k-1}{k-j}}{\binom{n+k-1}{k}}(-1)^j,\quad j=0,\ldots,k,
    \end{equation}
    where $a_0$ is free.
    In particular, if we choose $a_0=\binom{n+k-1}{k}$ we get
    $$p(x,y)=\sum_{j=0}^k\binom{k}{j}\binom{n+k-1}{k-j}(-1)^jx^{k-j}y^j.$$
    If $x=[z_0:\ldots:z_n]\in\Sigma_0$, then $|z_0|^2=r^2=\frac{1}{2n}=1-\sum_{l=1}^n|z_l|^2$, so
    $$P_k(x)=\sum_{j=0}^k\binom{k}{j}\binom{n+k-1}{k-j}(-1)^jr^{2(k-j)}(1-r^2)^j,$$
    where $r^2=\frac{1}{2n}$. That is, the value of $P_k$ does not depend on $x\in\Sigma_0$.
\end{proof}
\begin{cor}\label{Cor formula of ck}
    For $f\in\mathcal{H}_k$ we have $\F(f)=c_kf$, where
    $$c_k=c_{k,n}=\frac{\F(f_k)(\sigma_0)}{f_k(\sigma_0)}=|\Sigma_{\sigma_0}|\sum_{j=0}^k\binom{k}{j}\binom{n+k-1}{k-j}(-r^2)^{k-j}(1-r^2)^j,$$
    where $|\Sigma_{\sigma_0}|$ is the area of $\Sigma_{\sigma_0}\subseteq\CP^n$ and $r^2=\frac{1}{2n}$.
\end{cor}

\begin{proof}[Proof of Theorem \ref{Thm B for CPn with 2n-1 prime}]
    By Lemma \ref{lem f must be in kernel of F}, we know that $f:=\frac{d}{dt}\big|_{t=0}\psi_t\in\ker(\F)$.
    We will show that $\ker(\F)=\{0\}$.

    By Schur's Lemma, $\F$ diagonalizes in the spectral decomposition of 
    $$L^2(\CP^n)=\bigoplus_{k=0}^\infty\mathcal{H}_k,$$ 
    and the eigenvalues $c_k$ are given by Corollary \ref{Cor formula of ck}. 
    Hence, $\ker(\F)=\{0\}$ if and only if $c_k\neq0$ for any $k\geq0$. Equivalently, it is enough to show that
    \begin{equation}
        S_k:=\frac{c_k(-1)^k(2n)^k}{|\Sigma_{\sigma_0}|}=\sum_{j=0}^k\binom{k}{j}\binom{n+k-1}{k-j}(-p)^j\neq0,
    \end{equation}
    where $p=2n-1$ is prime.

    For $m\in\N$, let $v_p(m)\in\N$ the unique nonnegative integer such that $p^{v_p(m)}$ divides $m$ but $p^{v_p(m)+1}$ does not.
    Consider 
    $$s_j:=\binom{k}{j}\binom{n+k-1}{k-j}(-p)^j,$$ 
    we claim that
    \begin{equation}\label{eq involving vp}
        v_p(s_j)\geq v_p(s_0)+1,\quad \forall j\geq1.
    \end{equation}
    Proving (\ref{eq involving vp}) proves that $S_k\neq0$ since in this case
    $$S_k=\sum_{j=0}^ks_j\equiv s_0\not\equiv0,\quad(\text{mod } p^{v_p(s_0)+1}).$$

    Let us prove (\ref{eq involving vp}). 
    Observe that
    $$s_j=\binom{k}{j}^2\binom{n+k-1}{k}\binom{j+n-1}{j}^{-1}(-p)^j,$$
    so
    \begin{align*}
        v_p(s_j)&=2v_p\Big(\binom{k}{j}\Big)+v_p\Big(\binom{n+k-1}{k}\Big)-v_p\Big(\binom{j+n-1}{j}\Big)+j\\
        &=2v_p\Big(\binom{k}{j}\Big)+v_p(s_0)-v_p\Big(\binom{j+n-1}{j}\Big)+j\\
        &\geq0+v_p(s_0)-v_p\Big(\binom{j+n-1}{j}\Big)+j.
    \end{align*}
    Hence, in order to prove (\ref{eq involving vp}), it is enough to show
    \begin{equation}\label{eq j geq vp(combinatorial)+1}
        j\geq v_p\Big(\binom{j+n-1}{j}\Big)+1,\quad\forall j\geq1.
    \end{equation}
    Observe that
    $$\binom{j+n-1}{j}=\frac{(j+1)(j+2)\ldots(j+n-1)}{n!},$$
    and since $p=2n-1>n$ is prime, $v_p(n!)=0$, so
    $$v_p\Big(\binom{j+n-1}{j}\Big)=\sum_{i=1}^{n-1}v_p(j+i).$$
    Moreover, since $p>n-1$, at most one of the terms in $\{j+i\}_{i=1}^{n-1}$ is divisible by $p$.
    If none of them is divisible by $p$ then (\ref{eq j geq vp(combinatorial)+1}) is trivial.
    Assume that $j+i_0$ is divisible by $p$, so $v_p(\binom{j+n-1}{j})=v_p(j+i_0)=:v\geq1$.
    Write $j+i_0=p^{v}a$, where $p$ does not divide $a$.
    Observe that
    $$j-1=p^va-i_0-1\geq p^v-(n-1)-1=p^v-\frac{p+1}{2}.$$
    Consider $f(x)=x^v-\frac{x+1}{2}$, and notice that $f'(x)>0$ for $x\geq3$.
    Since $p$ is an odd prime then $p\geq 3$, so
    $$j-1\geq f(p)\geq f(3)=3^v-2\geq v,\quad v\in\N,$$
    where the last inequality comes from an elementary induction. This proves (\ref{eq j geq vp(combinatorial)+1}) and concludes the proof.
\end{proof}

\section{Appendix}\label{section appendix}

This Appendix is separated into two subsections; the first one is to adapt a well-known result in the literature, and the second is dedicated to proving a technical lemma required in this work.

\subsection{A metric relation with Killing tensors}

\text{ }

Hangan \cite{HanganMinimalHyperplanes} characterization of metrics in $\R^n$ with minimal hyperplanes, and Ambrozio-Marques-Neves \cite{AmbrozioMinimalEquators} classification of metrics in $\S^n$ with minimal equators depend on Killing tensors. 
In this subsection, we present the precise technical step where such tensors appear.

\text{ }

The next result is deduced from results of Section $2$ of \cite{AmbrozioMinimalEquators}.
For completeness, we present the sketch of the proof.

\begin{lem}\label{lemma T(X,X,X)=g(X,X)tr(T)(X) then k Killing}
    Consider an oriented manifold $M^n$ of dimension $n$. 
    Assume that $g$ and $\overline{g}$ are Riemannian metrics in $M^n$, and let
    $$T(X,Y,Z):=g(\nabla^g_XY-\overline{\nabla}_XY,Z),\quad X,Y,Z\in TM$$
    where $\nabla^g$ and $\overline{\nabla}$ are the Levi-Civita connections associated to $g$ and $\overline{g}$, respectively. Then the following conditions are equivalent
    \begin{itemize}
        \item[$i)$] $T(X,X,X)=g(X,X)tr_g^{12}(T)(X)$ for all $X\in TM$;
        \item[$ii)$] The Riemannian metric $k:=F^{-1}g$ satisfies the Killing equation
        $$\overline{\nabla}k(X,X,X)=0,$$
        where $F$ is determined by relating the volume forms associated to $g$ and $\overline{g}$ as $dV_g=F^{\frac{n+1}{4}}dV_{\overline{g}}$.
    \end{itemize}
    
\end{lem}
\begin{proof}
    Let $S_3$ be the group of permutations of the set $\{1,2,3\}$.
    For a three tensor $L(X_1,X_2,X_3)$ in $M^n$, consider 
    $$L^S(X_1,X_2,X_3)=L(X_1,X_2,X_3)+L(X_2,X_3,X_1)+L(X_3,X_1,X_2),$$
    the cyclic symmetrization of $L$.
    
    As proved in Lemma 2.4 of \cite{AmbrozioMinimalEquators}, we always have $T^S=\frac{1}{2}(\overline{\nabla}g)^S$ and
    $$tr_g^{23}(T)=d(\log(\psi)),$$
    where $\psi:M^n\rightarrow\R_{>0}$ is determined by the volume forms of $g$ and $\overline{g}$ with the relation $dV_g=\psi dV_{\overline{g}}$.

    Consider the tensor $\mathcal{S}:=T-g\otimes tr_g^{12}(T)$, and observe that $\mathcal{S}$ is symmetric in the first two entries.
    Using the identities proved in Proposition 2.5 of \cite{AmbrozioMinimalEquators}, we can verify
    \begin{align*}
        tr_g^{12}(\mathcal{S}^S)-2d\log(\psi)&=tr_g^{12}(\mathcal{S})+2tr_g^{23}(\mathcal{S})-2d\log(\psi)\\
        &=tr_g^{12}(T)-tr(g\otimes tr_g^{12}(T))-2tr_g^{23}(g\otimes tr_g^{12}(T))\\
        &=-(n+1)tr_g^{12}(T).
    \end{align*}
    Hence,
    \begin{align*}
        \mathcal{S}^S&=(T-g\otimes tr_g^{12}(T))^S\\
        &=\Big(\frac{1}{2}\overline{\nabla}g-\frac{2}{n+1}g\otimes d\log(\psi)+\frac{1}{n+1}g\otimes tr_g^{12}(\mathcal{S}^S)\Big)^S.
    \end{align*}
    If condition $(i)$ holds, then $\mathcal{S}(X,X,X)=0$ for any $X\in TM$, and so by polarization $\mathcal{S}^S=0$.
    This implies
    $$\Big(\overline{\nabla}g-\frac{4}{n+1}g\otimes d\log(\psi)\Big)^S=0.$$
    As shown in Theorem 2.6 of \cite{AmbrozioMinimalEquators}, the last equation is equivalent to $(ii)$
\end{proof}

\subsection{An algebraic result for quaternionic vector spaces}

\text{ }

In this section, we prove the final algebraic argument in the proof of Theorem \ref{Thm A CPn} for the case $\HP^n$.

\text{ }

Consider in $\R^{4n}=\HH^n$ the canonical inner product $\overline{g}$ and the quaternion structures $I,J,K:\R^{4n}\rightarrow\R^{4n}$.
For each $v\in\R^{4n}$ with $\overline{g}(v,v)=1$ consider $H_v:\R^{4n}\times\R^{4n}\rightarrow\R$ the symmetric operator $H_v(u,w)=\overline{g}(\hat{H}_v(u),w)$ where
$$\hat{H}_v(u):=(4n-1)\Big(\overline{g}(u,v)v+\sum_{\alpha\in\{I,J,K\}}\overline{g}(u,\alpha v)\alpha v\Big)-3u.$$

\begin{lem}\label{lemma tr(Hv)=Hv(Nv,Nv) for HPn}
    Let $k$ be an inner product in $\R^{4n}=\HH^n$, $n\geq2$.
    Assume that 
    \begin{equation}\label{eq lemma appendix tr)Hv=Hv(Nv,Nv)}
        tr_k(H_v)=H_v(N_v,N_v),\quad\forall v\in\mathbb{R}^{4n}\text{ with } \overline{g}(v,v)=1,
    \end{equation}
    where $N_v$ is determined (up to sign) by $k(N_v,N_v)=1$ and $k(N_v,w)=0$ for any $w$ with $\overline{g}(w,v)=0$.
    Then $k=\lambda^{-1}\overline{g}$ for some constant $\lambda>0$.
\end{lem}

\begin{proof}
    Write $k(X,Y)=\overline{g}(AX,Y)$ for some map $A:\R^{4n}\rightarrow \R^{4n}$ and let $B:=A^{-1}$. 
    Since $A$ is symmetric and positive definite, write $Be_i=\lambda_ie_i$ for some orthonormal basis $\{e_i\}_i$ of $\R^{4n}$ with $\lambda_i>0$ (orthonormal with respect to $\overline{g})$.
    We will express (\ref{eq lemma appendix tr)Hv=Hv(Nv,Nv)}) using this basis.
    
    Notice that $N_v=\frac{Bv}{\sqrt{\overline{g}(v,Bv)}}$, so the right-hand side of (\ref{eq lemma appendix tr)Hv=Hv(Nv,Nv)}) is given by
    \begin{align*}
        H_v(N_v,N_v)&=\frac{H_v(Bv,Bv)}{\overline{g}(v,Bv)}\\
        &=\frac{\overline{g}((4n-1)[\overline{g}(Bv,v)v+\sum_{\alpha\in\{I,J,K\}}\overline{g}(Bv,\alpha v)\alpha v]-3Bv,Bv)}{\overline{g}(v,Bv)}
    \end{align*}
    \begin{equation}\label{Hv(Nv,Nv) in HPn}
        H_v(N_v,N_v)=\frac{(4n-1)\Big[\overline{g}(Bv,v)^2+\sum_{\alpha\in\{I,J,K\}}\overline{g}(Bv,\alpha v)^2\Big]-3\overline{g}(Bv,Bv)}{\overline{g}(v,Bv)}.
    \end{equation}
    On the other hand, if $L_v\subseteq\R^{4n}$ is the quaternionic line determined by $v$, then the right-hand side of (\ref{eq lemma appendix tr)Hv=Hv(Nv,Nv)}) can be written as
    \begin{align*}
        tr_k(H_v)&=tr_{\overline{g}}(B\hat{H}_v)=tr_{\overline{g}}(B\hat{H}_v|_{Lv^{\perp}})+\overline{g}(B\hat{H}_vv,v)+\sum_{\alpha\in\{I,J,K\}}\overline{g}(B\hat{H}_v\alpha v,\alpha v)\\
        &=-3tr_{\overline{g}}(B|_{L_v^{\perp}})+(4n-4)[\overline{g}(Bv,v)+\sum_{\alpha\in\{I,J,K\}}\overline{g}(B\alpha v,\alpha v)],
    \end{align*}
    where $L^{\perp}_v$ is with respect to $\overline{g}$.
    Equivalently, the last expression is 
    \begin{equation}\label{tr(Hv) in HPn}
        tr_k(H_v)=(4n-1)\Big[g(Bv,v)+\sum_{\alpha\in\{I,J,K\}}\overline{g}(B\alpha v,\alpha v)\Big]-3tr_{\overline{g}}(B).
    \end{equation}
    Combining (\ref{eq lemma appendix tr)Hv=Hv(Nv,Nv)}) with (\ref{Hv(Nv,Nv) in HPn}) and (\ref{tr(Hv) in HPn}), we get a polynomial equation in $v$. 
    More precisely, using the factor $\overline{g}(v,v)=1$, we obtain the next homogeneous polynomial equation of degree $4$
    \begin{equation}
    \begin{split}
        \overline{g}(Bv,v)\Big((4n-1)\Big[\overline{g}(Bv,v)+\sum_{\alpha\in\{I,J,K\}}\overline{g}(B\alpha v,\alpha v)\Big]-3tr_{\overline{g}}(B)\overline{g}(v,v)\Big)\\
        =
        (4n-1)\Big[\overline{g}(Bv,v)^2+\sum_{\alpha\in\{I,J,K\}}\overline{g}(Bv,\alpha v)^2\Big]-3\overline{g}(Bv,Bv)\overline{g}(v,v).
    \end{split}
    \label{main eq lemma appendix}
\end{equation}
The last equation holds for any $v\in\R^{4n}$ since the equation is homogeneous.
    
    Using the coordinates $v=\sum_iv^ie_i$, we will determine the coefficients of the polynomial equation (\ref{main eq lemma appendix}).
    To simplify computations, let us introduce $Q=-\sum_{\alpha\in\{I,J,K\}}\alpha B\alpha$, and
    $Q_{ij}:=\overline{g}(Qe_i,e_j)$.
    In particular, $Q_{ij}=Q_{ji}$. 
    In the following, we use the Einstein sum convention. 
    
    Let us open the first line of (\ref{main eq lemma appendix}) for $v=v^ie_i$, that is
    \begin{equation*}\label{eq LHS of polynomial equation of degree 4}
        (\lambda_k(v^k)^2)\big((4n-1)[(\lambda_i(v^i)^2)+Q_{ij}v^iv^j]-3tr_{\overline{g}}(B)(v^i)^2\delta_{ij}),
    \end{equation*}
    the second line is determined by
    \begin{equation*}\label{eq RHS of polynomial equation of degree 4}
        (4n-1)[(\lambda_i(v^i)^2)^2+\sum_{\alpha\in\{I,J,K\}}(v^iv^j\lambda_i\overline{g}(e_i,\alpha e_j))^2]-3\lambda_i^2(v^i)^2((v^j)^2\delta_{jj}),
    \end{equation*}
    the last two identities must coincide, and determine an equality of homogeneous polynomials of degree $4$, so we must have equality of coefficients.
    
    First, the coefficient for $(v^i)^4$ gives the equation
    $$\lambda_i((4n-1)[\lambda_i+Q_{ii}]-3tr_{\overline{g}}(B))=(4n-1)\lambda_i^2-3\lambda_i^2,$$
    equivalently
    \begin{equation}\label{eq Q_ii}
        Q_{ii}=\frac{3tr_{\overline{g}}(B)-3\lambda_i}{4n-1},\quad\forall i.
    \end{equation}
    The coefficient of $(v^i)^2(v^j)^2$ for $i\neq j$ gives
    \begin{align*}
        (4n-1)[2\lambda_i&\lambda_j+\lambda_iQ_{jj}+\lambda_jQ_{ii}]-3tr_{\overline{g}}(B)(\lambda_i+\lambda_j)\\
        &=(4n-1)[2\lambda_i\lambda_j+(\lambda_i-\lambda_j)^2\sum_{\alpha\in\{I,J,K\}}\overline{g}(e_i,\alpha e_j)^2]-3\lambda_i^2-3\lambda_j^2,
    \end{align*}
    we can simplify the last equation using (\ref{eq Q_ii}) for $i$ and $j$ to we get
    $$(\lambda_i-\lambda_j)^2\Big[3-(4n-1)\sum_{\alpha\in\{I,J,K\}}\overline{g}(e_i,\alpha e_j)^2\Big]=0.$$
    To simplify notation, for $i\neq j$ consider in $\R^3$ the vector 
    $$e_{ij}:=\big(\overline{g}(e_i,\alpha e_j)\big)_{\alpha\in\{I,J,K\}}\in\R^3.$$
    With this notation the last equation is simplified to
    \begin{equation}\label{eq (lambda_i-lambda_j)^2(e-(4n-1)sum)=0}
        (\lambda_i-\lambda_j)^2(3-(4n-1)|e_{ij}|^2)=0.
    \end{equation}
    
    Let us determine the remaining coefficients.
    The coefficient of $(v^i)^3v^j$ for $j\neq i$ gives
    \begin{equation}\label{eq Q_ij=0}
        Q_{ij}=0,\quad i\neq j.
    \end{equation}
    Now, the coefficient $(v^i)^2v^jv^k$ with $i,j,k$ distinct we get the relation
    $$\lambda_i(Q_{kj}+Q_{jk})=2(\lambda_i-\lambda_j)(\lambda_i-\lambda_k)\sum_{\alpha\in\{I,J,K\}}\overline{g}(e_i,\alpha e_j)\overline{g}(e_i,\alpha e_k),$$
    and using (\ref{eq Q_ij=0}) we get
    \begin{equation}\label{eq appendix (lambda-lambda)(lambda-lambda)eij,eik}
        (\lambda_i-\lambda_j)(\lambda_i-\lambda_k)\inner{e_{ij}}{e_{ik}}=0,\quad i,j,k\text{ distinct}.
    \end{equation}
    Finally, the coefficient of $v^iv^jv^kv^l$ for $i,j,k,l$ distinct we get
\begin{equation}
    \begin{split}
        0=&(\lambda_i-\lambda_j)(\lambda_k-\lambda_l)\inner{e_{ij}}{e_{kl}}+(\lambda_i-\lambda_k)(\lambda_j-\lambda_l)\inner{e_{ik}}{e_{jl}}\\
        &+(\lambda_i-\lambda_l)(\lambda_j-\lambda_k)\inner{e_{il}}{e_{jk}},\quad i,j,k,l\text{ distinct}.
    \end{split}
    \label{eq lemma appendix coefficient i,j,k,l}
\end{equation}

Next, we will use (\ref{eq (lambda_i-lambda_j)^2(e-(4n-1)sum)=0}), (\ref{eq appendix (lambda-lambda)(lambda-lambda)eij,eik}), and (\ref{eq lemma appendix coefficient i,j,k,l}) to prove that $\lambda_i=\lambda_j$ for any $i,j$.
To do so, we prove a series of claims. 

First, decompose $\R^{4n}=\bigoplus_{i=1}^NE_i$ in the eigenspaces of $B$, and $m_i=\dim(E_i)$ the multiplicity of each eigenvalue. We assume that $m_1=\max_i\{m_i\}$.

\text{ }

\textbf{Claim 1.} $N\leq4$.

Indeed, let $e_i\in E_i$, we have that $\{e_{1i}\}_{i=2}^N\subseteq\R^3$ is a set of non-zero vectors by (\ref{eq (lambda_i-lambda_j)^2(e-(4n-1)sum)=0}). However, by (\ref{eq appendix (lambda-lambda)(lambda-lambda)eij,eik}) they are also orthogonal, so $N-1\leq3$.

\text{ }

\textbf{Claim 2.} If $m_1\geq4$ then $m_1=4n$, and so $N=1$.

Suppose by contradiction that $N\geq2$. Let $\{e_1^1,\ldots e_1^{m_1}\}$ be an orthogonal basis of $E_1$, and $e_2\in E_2$ a unit vector.
In this case
$$e^k_{21}:=\big(\overline{g}(e_2,\alpha e_1^k)\big)_{\alpha\in\{I,J,K\}}\in\R^3,$$
is non-zero by (\ref{eq (lambda_i-lambda_j)^2(e-(4n-1)sum)=0}). Moreover, by (\ref{eq appendix (lambda-lambda)(lambda-lambda)eij,eik}) the set $\{e^k_{21}\}_{k=1}^{m_1}$ is also orthogonal, which is a contradiction since $\dim(\R^3)=3<4\leq m_1$.

\text{ }

\textbf{Claim 3.} $m_1\neq3$.

Suppose by contradiction that $m_1=3$, in particular $N\geq2$. 
We cannot have $N=2$ otherwise $8\leq 4n=m_1+m_2\leq2m_1=6$.
Hence, $N\geq3$.
Consider an orthogonal basis $\{e_1,e_2,e_3\}$ of $E_1$, and take unit vectors $e_4\in E_2$ and $e_5\in E_3$.
In this case, $\mathcal{B}:=\{e_{41},e_{42},e_{43}\}$ is an orthogonal basis of $\R^3$ (this follows from (\ref{eq (lambda_i-lambda_j)^2(e-(4n-1)sum)=0}) and (\ref{eq appendix (lambda-lambda)(lambda-lambda)eij,eik})).
However, $e_{45}\in\R^3$ is non-zero by (\ref{eq (lambda_i-lambda_j)^2(e-(4n-1)sum)=0}), but $e_{45}$ is also orthogonal to the elements of $\mathcal{B}$ by (\ref{eq appendix (lambda-lambda)(lambda-lambda)eij,eik}), which proves that $e_{45}$ is zero, a contradiction. 

\text{ }

\textbf{Claim 4.} $m_1\neq1,2$

Assume by contradiction that $m_1\leq2$. 
In this case, by Claim 1 we get
$$8\leq 4n=\sum_{i=1}^Nm_i\leq Nm_1\leq 4\cdot2=8,$$
so $N=4$ and $m_i=2$ for any $i$.
Let $\{e_{2i-1},e_{2i}\}$ be an orthogonal basis of $E_i$.
In this case $\{e_{1k}\}_{k=3}^8$ is a set of non-zero vectors of $\R^3$ (by (\ref{eq (lambda_i-lambda_j)^2(e-(4n-1)sum)=0}). However, by (\ref{eq appendix (lambda-lambda)(lambda-lambda)eij,eik}) they are also orthogonal, which is a contradiction.

By Claims 1 to 4, we have proved that $N=1$ and $m_1=4n$, that is, $B=\lambda \text{Id}$.
In this case, since $A=B^{-1}$, we have $k(X,Y)=\overline{g}(AX,Y)=\lambda^{-1}g(X,Y)$.
\end{proof}

\printbibliography

@article{AMN,
author = {Ambrozio, L. and C. Marques, F. and Neves, A.},
title = {\textit{Riemannian metrics on the sphere with Zoll families of minimal hypersurfaces}},
year = {2025},
journal = {J. Differ. Geom.},
pages = {269-341},
volume = {\textbf{130}},
}

@article{AGequiv,
author = {Ambrozio, L. and Guajardo, D.},
title = {\textit{Equivariant constructions of spheres with Zoll families of minimal spheres}},
year = {2026},
journal = {Adv. Math. },
volume = {\textbf{496}},
pages = {110989},
}

@article {Funk,
    AUTHOR = {Funk, P.},
     TITLE = {\textit{\"Uber {F}l\"achen mit lauter geschlossenen geod\"atischen
              {L}inien}},
   JOURNAL = {Math. Ann.},
  FJOURNAL = {Mathematische Annalen},
    VOLUME = {\textbf{74}},
      YEAR = {1913},
    NUMBER = {no. 2},
     PAGES = {278-300},
}

@article {Kiyohara,
    AUTHOR = {Kiyohara, K.},
     TITLE = {\textit{On infinitesimal {$C\sb{2\pi }$}-deformations of standard
              metrics on spheres}},
   JOURNAL = {Hokkaido Math. J.},
  FJOURNAL = {Hokkaido Mathematical Journal},
    VOLUME = {\textbf{13}},
      YEAR = {1984},
    NUMBER = {(2)},
     PAGES = {151--231},
}

@article{Zoll,
author = {Zoll, O.},
year = {1903},
pages = {108-133},
title = {\textit{Über Flächen mit Scharen geschlossener geodätischer Linien}},
volume = {\textbf{57}},
journal = {Math. Ann.},
}

@article{Bes,
author= {Besse, A.},
year = {1978},
title = {\textit{Manifolds all of whose geodesics are closed}},
journal = {Ergebnisse der Mathematik
und ihrer Grenzgebiete},
volume = {\textbf{93}},
}

@article{Guillemin,
author = {Guillemin, V.},
year = {1976},
pages = {85-119},
title = {\textit{The Radon transform on Zoll surfaces}},
volume = {\textbf{22}},
journal = {Advances in Math.},
number = { no. 1},
}

@article{Gre,
author= {Green, L. W.},
year = {1963},
title = {\textit{Auf Wiedersehensfläche}},
journal = {Ann. of Math.},
number = { (2)},
volume = {\textbf{78}},
pages = {289-299}
}

@article{Luzzi,
author= {Luzzi, L.},
title = {\textit{Balanced metrics, Zoll deformations and isosystolic inequalities in $\C \mathbb{P}^n$}},
journal = {Sel. Math. New Ser.},
year = {2026},
volume = {32},
pages = {40}
}

@article{GrinbergHarmonicsProjective,
author = {Grinberg, E. L.},
year = {1983},
pages = {187-203},
title = {{\it Spherical Harmonics and Integral Geometry on Projective Spaces}},
journal = {Trans. Am. Math. Soc.},
volume ={\textbf{1}}
}

@article{HanganMinimalHyperplanes,
author = {Hangan, T.},
year = {1996},
pages = {326-334},
title = {{\it On the Riemannian metrics in $\mathbb{R}^n$ which admit all hyperplanes as minimal hypersurfaces}},
journal = {J. Geom. Phys.},
volume ={\textbf{18}},
}

@article{AmbrozioMinimalEquators,
author= {Ambrozio, L.},
title = {\textit{Spheres with minimal equators}},
journal = {São Paulo J. Math. Sci.},
year = {2026},
volume = {20},
pages = {25}
}

@article{GotohTCayley,
author= {Gotoh, T.},
title = {\textit{The nullity of a compact minimal hypersurface in a compact symmetric space of rank one}},
journal = {Hokkaido Math. J.},
year = {2004},
volume = {33},
pages = {429-441}
}

@article{MatveevNikolayevskyKillingNODecomposable,
author= {Matveev, V. and Nikolayevsky, Y.},
title = {\textit{Quadratic Killing tensors on symmetric spaces which are not
generated by Killing vector fields}},
journal = {C. R. Math. Acad. Sci. Paris},
year = {2024},
volume = {\textbf{362}},
pages = {1043–1049}
}

@article{GustavoZollcharact,
author= {Martins, G.},
title = {\textit{Topological and spectral rigidity of hypersurface Zoll manifolds}},
journal = {preprint arXiv:2609.20689},
year = {},
volume = {},
pages = {}
}

@book{Helgason1999Radon,
  author    = {Helgason, Sigurdur},
  title     = {The Radon Transform},
  edition   = {2nd},
  publisher = {Birkh{\"a}user},
  address   = {Boston, MA},
  year      = {1999},
  doi       = {10.1007/978-1-4757-1463-0}
}

\end{document}